\documentclass{amsart}

\pdfoutput=1

\usepackage[utf8]{inputenc}
\usepackage[english]{babel}
\usepackage{lmodern}
\usepackage[T1]{fontenc}

\usepackage[colorlinks=true,urlcolor=blue,citecolor=red,linkcolor=blue,anchorcolor=red]{hyperref} 

\usepackage{amsfonts,amssymb,amsthm,amsmath}
\usepackage{mathtools}
\usepackage[all]{xy}

\usepackage{url}
\usepackage[final,babel,protrusion=false]{microtype}
\usepackage[]{enumerate}
\usepackage[]{enumitem}

\usepackage{setspace}
\newcommand{\fR}{\mathbb{R}}
\newcommand{\fZ}{\mathbb{Z}}
\newcommand{\fN}{\mathbb{N}}
\newcommand{\fC}{\mathbb{C}}

\newcommand{\g}{\mathfrak{g}}

\newcommand{\h}{\mathfrak{h}}
\newcommand{\p}{\mathfrak{p}}
\renewcommand{\k}{\mathfrak{k}}
\newcommand{\n}{\mathfrak{n}}

\renewcommand{\a}{\mathfrak{a}}
\newcommand{\frb}{\mathfrak{b}}
\renewcommand{\l}{\mathfrak{l}}
\newcommand{\m}{\mathfrak{m}}

\DeclareMathOperator{\C}{C} 
\DeclareMathOperator{\ad}{ad}
\DeclareMathOperator{\Ad}{Ad}
\DeclareMathOperator{\Id}{Id}

\DeclareMathOperator{\End}{End}

\DeclareMathOperator{\Exp}{Exp}

\DeclareMathOperator{\Lie}{Lie}

\DeclareMathOperator{\Diag}{Diag}

\newtheorem{thm}{Theorem}[section]
\newtheorem{prop}[thm]{Proposition}
\newtheorem{cor}[thm]{Corollary}
\newtheorem{lem}[thm]{Lemma}
\newtheorem{defn}[thm]{Definition}

\newtheorem*{defn*}{Definition}
\newtheorem*{thm*}{Theorem}
\newtheorem*{prop*}{Proposition}
\newtheorem*{cor*}{Corollary}

\theoremstyle{definition}
\newtheorem{ex}[thm]{Example}

\newcommand{\exend}{
    \hbox{}\nobreak\hfill\ensuremath{\lhd}
}

\title{Strongly exponential symmetric spaces}
\author{Yannick Voglaire}

\address{
Department of Mathematics\\
The Pennsylvania State University\\
104 McAllister Bldg\\
University Park, 16802 PA, USA\\
}
\email{\url{yannick.voglaire@gmail.com}}

\thanks{Part of this work was done during the author's PhD thesis at the Université Catholique de Louvain, Belgium, and the Université de Reims, France, under the supervision of Pierre Bieliavsky and Michael Pevzner.
The author thanks Mathieu Carette for discussions on midpoints and geodesics in symmetric spaces.
This work was partially supported for travel by the Belgian Science Policy (belspo) through the Interuniversity Attraction Pole (IAP) PHASE VII/18 project ``DYGEST''. \\
\indent 
The published version of this paper is available at \url{http://imrn.oxfordjournals.org/cgi/content/abstract/rnt149?ijkey=My6gILrUCEG3Z1z&keytype=ref}.}

\hypersetup{
	pdfinfo={
		Title={Strongly exponential symmetric spaces},
		Subject={Mathematics},
		Keywords={Symmetric spaces, Lie triple systems, exponential map, midpoints, Mathematical physics, Quantization, triangles},
		Author={Yannick Voglaire},
		Date={2013},
	}
}

\begin{document}

\begin{abstract}%
We study the exponential map of connected symmetric spaces and characterize, in terms of midpoints and of infinitesimal conditions, when it is a diffeomorphism, generalizing the Dixmier--Saito theorem for solvable Lie groups. 
We then give a geometric characterization of the (strongly) exponential solvable symmetric spaces as those spaces for which every triangle admits a unique double triangle.
This work is motivated by Weinstein's quantization by groupoids program applied to symmetric spaces.
\end{abstract}

\maketitle

\section{Introduction}

In the late 1950's, Dixmier \cite{dixmier_application_1957} and Saito \cite{saito_sur_1957} characterized the \emph{strongly exponential Lie groups}: the simply connected solvable Lie groups for which the exponential map is a global diffeomorphism. 
They are characterized by the fact that the exponential map is either injective or surjective, or equivalently the fact that the adjoint action of their Lie algebra has no purely imaginary eigenvalue.
In contrast, the exponential map of semisimple Lie groups is never injective, but may be surjective or have dense image.
As a result, the latter two weaker notions of exponentiality have been studied for general Lie groups by a number of authors 
(see \cite{dbarokovi_surjectivity_1997,wuestner_short_2001} for two surveys and for references).
Recently, a number of their results were generalized to symmetric spaces in \cite{rozanov_exponentiality_2009}.
The first main result of the present paper intends to complete the picture by giving a version of the Dixmier--Saito theorem for symmetric spaces (see Theorems~\ref{thm:A}, \ref{thm:B} and \ref{thm:D}).

The second main result gives a geometric characterization of the strongly exponential symmetric spaces which are homogeneous spaces of solvable Lie groups, in terms of existence and uniqueness of \emph{double triangles} (see Theorem~\ref{thm:C}).
This is related to a beautiful theorem of A.~Weinstein \cite{weinstein_traces_1994} relating the graph of multiplication in the pair groupoid of a symplectic symmetric space to the symplectic area of such double triangles.
Weinstein's theorem was developed in the context of his quantization by groupoids program and forms, together with the extension of Bieliavsky's quantization of symplectic symmetric spaces \cite{bieliavsky_strict_2002}, the initial motivation for the present work.

Since we will only be concerned with strong exponentiality in this paper, from now on \emph{we will drop the adjective ``strong''}, although nowadays exponentiality has a weaker meaning in the literature.

\medskip

A \emph{symmetric space} is a manifold $M$ endowed with a family of involutions $\{s_x\}_{x\in M}$ of $M$, called symmetries, 
such that $x$ is an isolated (not necessarily unique) fixed point of $s_x$, and such that $s_x(s_yz)=s_{(s_xy)}(s_xz)$ for all $x,y,z\in M$.
A \emph{morphism} of symmetric spaces $M,N$ is a smooth map $\phi:M\to N$ such that $\phi(s_xy)=s_{\phi(x)}\phi(y)$ for all $x,y\in M$.

Some basic examples to keep in mind are the following:
\begin{enumerate}[label=(Ex\arabic*),leftmargin=*]
\item the $n$-dimensional Euclidean space $\fR^n$ with symmetries $s_xy=2x-y$,
\item the $n$-sphere $S^n$ where the symmetry at $x$ is the restriction to $S^n$ of the axial symmetry about $0x$,
\item the hyperbolic plane $\Pi_2$ with the Riemannian geodesic symmetries,
\item any Lie group $G$ with symmetries $s_gl=gl^{-1}g$,
\end{enumerate}
as well as two solvable two-dimensional examples:
\begin{enumerate}[resume,label=(Ex\arabic*),leftmargin=*]
\item \label{ex:2dim-solvable-exponential} the space $A_2=\fR^2$ with symmetries $s_{(a,b)}(a',b')=(2a-a',2\cosh(a-a')b-b')$,
\item \label{ex:2dim-solvable-non-exponential} the space $B_2=\fR^2$ with symmetries $s_{(a,b)}(a',b')=(2a-a',2\cos(a-a')b-b')$.
\end{enumerate}

The connected symmetric spaces in this sense (due to Loos \cite{loos_symmetric_1969}) coincide with the \emph{homogeneous symmetric spaces}: the homogeneous spaces $G/H$ of connected Lie groups $G$, with $H$ an open subgroup of the fixed points set $G^\sigma$ of an involutive automorphism $\sigma$ of $G$ (called an \emph{involution} in what follows).
The symmetries in that case are given by $s_{gH}lH=g\sigma(g^{-1}l)H$.
Such a triple $(G,H,\sigma)$ will be called a \emph{symmetric triple}. 
When $H$ is the full fixed point subgroup $G^\sigma$, we simply refer to the symmetric pair $(G,\sigma)$. 
A symmetric triple $(G,H,\sigma)$ is said to \emph{realize} a symmetric space $M$ if $G/H\cong M$ as symmetric spaces.

Any connected symmetric space $M$ admits a ``minimal'' realization as a homogeneous symmetric space of its \emph{transvection group}. 
The latter is defined as the subgroup of the automorphism group of $M$ generated by all the products of an even number of symmetries. 
It is a finite dimensional Lie group which is transitive on $M$, stabilized by the conjugation (in the automorphism group) by $s_o$, for any $o\in M$, and minimal for these properties.
A symmetric space is said to be \emph{semisimple} (respectively, \emph{solvable}) if its transvection group is semisimple (respectively, solvable).

Every Lie group $G$ is a symmetric space when endowed with the symmetries $s_{g}l=gl^{-1}g$. It is usually realized by $(G\times G,\Diag(G),\sigma_G)$ with $\sigma_G(g,l)=(l,g)$, where $G\times G$ is in general larger than the transvection group of $G$.
Given an involution $\sigma$ on $G$, two canonical symmetric subspaces of $G$ are defined by $P=\{g\in G \mid \sigma(g)=g^{-1}\}$ and $G_\sigma=\{ g\sigma(g)^{-1} \mid g\in G\}$.
When $G$ is connected, $G_\sigma$ is the connected component of the identity of $P$, and is isomorphic to $G/G^\sigma$ through the \emph{quadratic representation} $Q:G/G^\sigma\to G_\sigma:gG^\sigma\to g\sigma(g)^{-1}$.

A symmetric subspace $N$ of a symmetric space $M\cong G/H$ is said to be \emph{normal} if there exists a $\sigma$-invariant normal subgroup $R\subset G$ such that $N\cong R/(R\cap H)$ through the isomorphism $M\cong G/H$.
Then $M$ is said to be a \emph{semi-direct product} $M_1\ltimes M_2$ of two symmetric subspaces $M_1$ and $M_2$ if $M$ is, as a smooth manifold, a direct product $M_1\times M_2$, and if $M_2$ is normal in $M$.

On the tangent space $\m$ at a given point of a symmetric space $M$, there is a structure of \emph{Lie triple system} $[\cdot,\cdot,\cdot]:\m^3\to \m$, generalizing that of Lie algebra for Lie groups. This trilinear bracket satisfies the properties:
\begin{enumerate}[label=(Lts\arabic*),leftmargin=*]
  \item \label{lts1} $[X,X,Y]=0$,
  \item \label{lts2} $[X,Y,Z]+[Y,Z,X]+[Z,X,Y]=0$,
  \item \label{lts3} $[X,Y,[U,V,W]] = [[X,Y,U],V,W] + [U,[X,Y,V],W] + [U,V,[X,Y,W]]$.
\end{enumerate}
Any Lie algebra $\g$ is a Lie triple system with bracket
\begin{align}
\label{eq:Lie-algebra-LTS}
[X,Y,Z]=[[X,Y],Z], 
\end{align}
and if $G$ is a Lie group with Lie algebra $\g$, this structure makes it the Lie triple system at $e$ of $G$ seen as a symmetric space.
More generally, if $\sigma$ is an involution on $\g$ and if $\g=\m\oplus\h$ is the corresponding decomposition into $(-1)$ and $(+1)$-eigenspaces of $\sigma$, then the bracket \eqref{eq:Lie-algebra-LTS} on $\g$ restricts to $\m$ as a Lie triple system, and any Lie triple system is of that kind.
The Lie triple system at the base point $eH$ of a symmetric space $G/H$ coming from a triple $(G,H,\sigma)$ is given by the bracket \eqref{eq:Lie-algebra-LTS} on the $(-1)$-eigenspace $\m$ of the involution $d\sigma_e$ on $\Lie(G)$.

Just as any connected symmetric space can be realized as a homogeneous symmetric space, any Lie triple system can be realized as the $(-1)$-eigenspace of a Lie algebra involution. The minimal such realization is called the \emph{standard embedding} and is defined (as a vector space) by $\g=\m\oplus[\m,\m]$ where $[\m,\m]$ denotes the span of the operators $L_{X,Y}=[X,Y,\cdot]$ for all $X,Y\in \m$. Conditions \ref{lts1}--\ref{lts3} show that the following equations define a Lie bracket on $\g$:
\begin{align*}
[X,Y] &= L_{X,Y}, \\
[L_{X,Y},Z] &= L_{X,Y}Z, \\
[L_{X,Y},L_{U,V}] &= L_{L_{X,Y}U,V} + L_{U,L_{X,Y}V},
\end{align*}
for all $X,Y,Z,U,V\in \m$. The involution $\sigma$ on $\g$ is $-\Id_\m\oplus\Id_{[\m,\m]}$ and one sees that $(\g,\sigma)$ is the infinitesimal data at the identity of the tranvection group.

On any symmetric space, there is a unique affine connection for which all the symmetries are affine transformations.
This connection is complete, torsion-free, and its curvature is parallel.
The curvature $R$ at a point $x$ on a symmetric space is related to its Lie triple system at that point by 
$R(X,Y)Z=-[X,Y,Z]$.
The exponential map at $x$ is denoted $\Exp_x:T_xM\to M$.
On a symmetric space $G/H$ with Lie triple system $\m\subset\g$ at $eH$, the set $\exp\m\subset G$ is an open submanifold of $P$, so that $\exp\m\subset G_\sigma \subset P$ are open inclusions.

On a symmetric space with a fixed base point $o$, there is a product $\perp$ defined by $x\perp y = s_xs_oy$.
If $M\cong G/H$, one proves that $s_{\Exp_o\frac{X}{2}}s_o=\exp X\in G$ for all $X\in \m$, so that the product satisfies $\Exp_o\frac{X}{2}\perp y=\exp(X)y$. 

A point $z$ is said to be a \emph{midpoint} of $x$ and $y$ if $s_zx=y$.
Our first main result relates the uniqueness of midpoints to the exponentiality of a symmetric space.

\begin{thm}
\label{thm:A}
Let $M$ be a connected symmetric space and $(G,H,\sigma)$ be any symmetric triple realizing it. Denote $[\cdot,\cdot,\cdot]$ its Lie triple system at some base point $o$, and $\m$ the $(-1)$-eigenspace of $\sigma$ in $\Lie(G)$. Then the following conditions are equivalent:
\begin{enumerate}[label=(\arabic*)]
  \item \label{thm:A:at-most-one-midpoint} any two points in $M$ have at most one midpoint;
  \item \label{thm:A:exp-inj} the exponential map at $o$ is injective;
  \item \label{thm:A:exp-diffeo} the exponential map at $o$ is a global diffeomorphism;
  \item \label{thm:A:M-loc-exp} $M$ is simply connected and no operator $Y\mapsto[Y,X,X]$, for $X\in T_oM$, has strictly negative eigenvalues;
  \item \label{thm:A:G-sigma-loc-exp} $M$ is simply connected and no operator $\ad X$, for $X\in \m$, has purely imaginary eigenvalues.  \qedhere
\end{enumerate}
\end{thm}

In particular, if every two points have at most one midpoint, then the space is exponential so that they actually have exactly one midpoint, and the map $\gamma:M\times M\to M$ sending two points to their unique midpoint is smooth.

By considering the symmetric space $M=(G\times G)/G$ associated to a connected Lie group $G$, in which case $[X,Y,Z]=[[X,Y],Z]$ for all $X,Y,Z\in\Lie(G)$, we recover the Dixmier--Saito theorem about the exponential map of solvable Lie groups.

A symmetric space satisfying any of the conditions of Theorem~\ref{thm:A} is said to be \emph{exponential}.

Examples \ref{ex:2dim-solvable-exponential} and \ref{ex:2dim-solvable-non-exponential} represent the two isomorphism classes of non-flat solvable two-dimensional symmetric spaces, and are realized by $(G,\sigma)$ where $G$ is respectively $SO(1,1)\ltimes \fR^2$ and the universal cover of $SO(2)\ltimes \fR^2$, and $\sigma$ is the lift to $G$ of the Lie algebra automorphism 
$-\Id\times\sigma_2$, with $\sigma_2(x,y)=(y,x)$.
It is easily seen that the first is exponential while the second is not.

\begin{thm}
\label{thm:B}
If $M$ is an exponential symmetric space, $G$ is its transvection group and $\sigma$ the conjugation by $s_o$ for some $o\in M$, then $G^\sigma$ is connected, $M$ is realized by $(G,G^\sigma,\sigma)$, and
 \begin{enumerate}[label=(\arabic*)]
   \item \label{item:semisimplemidpointmap-implies-cartaninvolution} $M$ is semisimple if and only if it is a Riemannian symmetric space of the non\-compact type,
   \item \label{item:solvablemidpointmap-implies-exists-exponentialsymmetricpair} $M$ is solvable if and only if $G$ is an exponential Lie group,
   \item \label{item:midpointmap-semidirect-product} $M$ is a semi-direct product $M_1\ltimes M_2$ with $M_1$ and $M_2$ respectively of the type \ref{item:semisimplemidpointmap-implies-cartaninvolution}\ and \ref{item:solvablemidpointmap-implies-exists-exponentialsymmetricpair}, and the map $\m_1\times \m_2\to M:(X,Y) \mapsto \Exp_o \frac{X}{2} \perp \Exp_o Y$ is a diffeomorphism. \qedhere
 \end{enumerate}
\end{thm}

In the solvable case, we can further characterize exponentiality in terms of \ref{ex:2dim-solvable-non-exponential}.
\begin{thm}
\label{thm:D}
Let $M$ be a solvable, connected and simply connected symmetric space. 
The conditions of Theorem~\ref{thm:A} are further equivalent to
\begin{enumerate}
\item[(6)] no factor space of $M$ has a subspace isomorphic to \ref{ex:2dim-solvable-non-exponential}. \qedhere
\end{enumerate}
\end{thm}

We now turn to our last theorem.
In an exponential symmetric space, a triple of points $(x,y,z)$ is called a \emph{triangle}, since any two points are joined by a unique geodesic.
A \emph{double triangle} of $(x,y,z)$ is a triangle $(a,b,c)$ whose midpoints of the edges are $x,y,z$. More precisely, $s_xc=a$, $s_ya=b$ and $s_zb=c$, so that $s_zs_ya=s_xa$.

\begin{thm}
\label{thm:C}
Let $M$ be an exponential symmetric space with transvection group $G$. Then every triangle has a unique double triangle if and only if $M$ is solvable. In that case, for any $g\in G$ and any $z\in M$, there exists a unique $x\in M$ such that $g\cdot x = s_z x$.
\end{thm}

Since any element of the transvection group is a product of an even number of symmetries, the second statement is actually a generalization of the first statement from triangles to $n$-gons, with odd $n$.
This is to be related to the fact that in the affine space $\fR^n$, any vector $v$ (representing a translation, and thus a general transvection) can be ``placed by its midpoint'' anywhere: given a point $z$, the translation of $x=z-\frac{1}{2}v$ by $v$ is the same as the image of $x$ by the symmetry at $z$.
This is not always possible in the hyperbolic plane for example, where if $g$ is too large, there are points at which $g$ cannot be ``placed by its midpoint''.

This idea and the notion of double triangle appear in the study of generating functions for canonical tranformations on symplectic symmetric spaces \cite{weinstein_traces_1994,rios_variational_2004,severa_symplectic_2006}.

The paper is organized as follows.
In Section~\ref{sec:midpoints-and-exponential-spaces}, we prove Theorems~\ref{thm:A}, \ref{thm:B} and \ref{thm:D}, and in Section~\ref{sec:double-triangles-and-solvable-spaces}, we prove Theorem~\ref{thm:C}.

\section{Midpoints and exponential spaces}
\label{sec:midpoints-and-exponential-spaces}

In this section, $M$ will always denote a connected symmetric space, $o\in M$ a fixed base point, and $\m$ the Lie triple system of $M$ at $o$.
For a symmetric triple $(G,H,\sigma)$ realizing $M$, we will denote $\g$ the Lie algebra of $G$, and $\g=\h\oplus\m$ its decomposition into $(\pm 1)$-eigenspaces for $\sigma$, identifying the $(-1)$-eigenspace with the Lie triple system.
A symmetric pair $(G,\sigma)$ (respectively, a symmetric triple $(G,H,\sigma)$) realizing $M$ will be called a \emph{transvection pair} (respectively, a \emph{transvection triple}) if $G$ is the transvection group of $M$, and $\sigma$ is the conjugation by $s_o$. 

\begin{defn}
A point $z\in M$ is said to be a \emph{midpoint} for $x,y\in M$ if $s_zx=y$,
and is said to be a \emph{square root} of $x$ if $z$ is a midpoint for $x$ and $o$.
\end{defn}

In order to motivate the definition of local exponentiality, let us recall
Helgason's formula and the well-known fact that it implies.

\begin{prop}[Helgason's formula {\cite[Theorem 4.1]{helgason_differential_1978}}]
\label{prop:helgason-diff-exp}
Let $(G,H,\sigma)$ be a symmetric triple, and for $g\in G$ denote $\tau(g)$ the map $hH\mapsto ghH$ from $G/H$ onto itself. The differential of the exponential map $\Exp_o$ from $\m$ to $G/H$ is given by
\begin{equation*}
d({\Exp_o})_X = d({\tau(\exp X)})_{eH}\circ \frac{\sinh{\ad X}}{\ad X} . 
\end{equation*} 
\end{prop}

\begin{lem}
\label{lem:explocaldiffeo-SpecadX-ipiZ}
Let $(G,H,\sigma)$ be a symmetric triple realizing $M$. 
Then ${\Exp_o}$ is a local diffeomorphism at $X\in\m$ if and only if the set of eigenvalues of $\ad X$ only meets $i\pi\fZ$ at $0$.
\end{lem}
\begin{proof}
By Proposition~\ref{prop:helgason-diff-exp}, $\Exp_o$ is a local diffeomorphism at $X\in\m$ if and only if $\det\left(\frac{\sinh(\ad X)}{\ad X} \right) \neq 0$, a condition which is equivalent to $\frac{\sinh \lambda}{\lambda} \neq 0$ for every eigenvalue $\lambda$ of $\ad X$, which in turn amounts to the statement.
\end{proof}

\begin{lem}
\label{lem:sympair-loc-exp-iff-M-loc-exp}
Let $(G,H,\sigma)$ be a symmetric triple realizing $M$, and let $X\in\m$. 
Then a complex number $\lambda$ is an eigenvalue of $\ad X$ if and only if its square is an eigenvalue of $Y\in \m\mapsto [Y,X,X]$.
\end{lem}
\begin{proof}
Seeing $\m$ as a subset of $\g$, we have $[\,\cdot\,,X,X]=[[\,\cdot\,,X],X]=(\ad X)^2|_\m$.
Set $m=\dim\m$ and $h=\dim\h$.
Choosing a basis of $\g=\m\oplus\h$, we may write $\ad X=\big(\begin{smallmatrix}0&A\\B&0\end{smallmatrix}\big)$ where $A$ and $B$ are $m\times h$ and $h\times m$ matrices respectively, and we have $(\ad X)^2=\big(\begin{smallmatrix}AB&0\\0&BA\end{smallmatrix}\big)$. By Sylvester's determinant formula, we get
$
\det(AB-\lambda I_m) = (-\lambda)^{m-h} \det(BA-\lambda I_k) 
$,
thus the eigenvalues of $AB=(\ad X)^2|_\m$ and $BA=(\ad X)^2|_\h$ are identical, up to $|m-h|$ zeros, which yields the conclusion.
\end{proof}

\begin{defn}
The symmetric space $M$ is said to be \emph{locally exponential} if no operator $\m\to\m:Y\mapsto[Y,X,X]$, $X\in \m$, has strictly negative eigenvalues, and \emph{exponential} if the exponential map $\Exp_o:T_oM\to M$ at $o$ is a global diffeomorphism.
\end{defn}

By translation, in an exponential symmetric space, every exponential map $\Exp_x$, $x\in M$, is a diffeomorphism.
Also, it is clear that a symmetric space is exponential if and only if any symmetric triple $(G,H,\sigma)$ realizing it is \emph{exponential} in the sense that the map $\exp:\m\to G/H$ is a global diffeomorphism.

\begin{lem}
\label{lem:at-most-one-midpoint-implies-exp-inj}
If every two points have at most one midpoint, then the exponential map $\Exp_o$ is injective.
\end{lem}
\begin{proof}
Assume $X,Y\in T_oM$ are such that $\Exp_oX=\Exp_oY$. 
Then, $\Exp_o\frac{X}{2}$ and $\Exp_o\frac{Y}{2}$ are both square roots of $\Exp_oX$, so that they coincide by assumption.
By induction, we obtain $\Exp_o\frac{X}{n}=\Exp_o\frac{Y}{n}$ for all positive integers $n$.
Since $\Exp_o$ is injective around $0$, we have $X=Y$.
\end{proof}

Recall that the automorphisms of a symmetric space are the affine automorphisms of its canonical affine structure. 
In particular, the exponential map at some point $o$ is equivariant under the actions of the 
subgroup of automorphisms fixing $o$.
\begin{lem}
\label{lem:exp-injective-implies-immersion-and-diffeotoRn}
If the exponential map $\Exp_o$ is injective, then it is an immersion, and $M$ is diffeomorphic to a Euclidean space. 
In particular, $M$ is simply connected and locally exponential.
\end{lem}
\begin{proof}
Let $(G,H,\sigma)$ be a triple realizing $M$.
Assume that $\Exp_o$ is not an immersion at $X\in\m$, so that $\ad X$ has an eigenvalue $k\pi i$ for some $k\in\fZ_*$.
We will show that there is a curve through $X$ in $\m$ on which $\Exp_o$ is constant, so that $\Exp_o$ is not injective.
Let $0\neq Y\in\g_\fC$ be such that $[X,Y]=k\pi iY$.
Spelling out the eigenvalue equation, one sees that $Y=(A_1+iB_1)+i(A_2+iB_2)$ with $A_j\in\m$, $B_j\in\h$, $[X,A_j]=k\pi B_j$, and $[X,B_j]=-k\pi A_j$, for $j=1,2$.
We can assume that $A_1\neq 0$ and define $X(t)=e^{t\ad B_1}X$ for all $t\in\fR$.
Then, on the one hand, $[B_1,X]=k\pi A_1\neq 0$ so $X(t)$ is not constant, and on the other hand
\begin{align*}
\frac{\sinh(\ad X(t))}{\ad X(t)} \dot X(t)
&= e^{t\ad B_1} \frac{\sinh(\ad X)}{\ad X} e^{-t\ad B_1} \dot X(t) \\
&= e^{t\ad B_1} \frac{\sinh(\ad X)}{\ad X} [B_1,X] \\
&= e^{t\ad B_1} \sin(k\pi) A_1 = 0 ,
\end{align*}
so that $\frac{d}{dt}\Exp_o(X(t))=0$ for all $t\in\fR$ by Helgason's formula.
Hence $t\mapsto X(t)$ is the curve we were looking for.

For the second assertion, in \cite[Theorem 1]{koh_affine_1965} and \cite[Chapter IV, Theorem 3.5]{loos_symmetric_1969}, it is shown that any connected symmetric space $G/H$ is smoothly fibered over a (connected) compact Riemannian symmetric subspace $N=K/(K\cap H)$ with fibers diffeomorphic to a Euclidean space. 
But no non-discrete compact symmetric space has injective exponential maps, so $N$ must be reduced to a point.

The last assertion follows now from the first two assertions and from Lemma~\ref{lem:explocaldiffeo-SpecadX-ipiZ}.
\end{proof}

\begin{lem}
\label{lem:exp-diffeo-implies-unique-midpoint}
Let $M$ be exponential.
Then every two points have a unique midpoint.
\end{lem}
\begin{proof}
Let $x,y\in M$, and $c:\fR\to M$ be the geodesic $c(t)=\Exp_x(t\Exp_x^{-1}(y))$. It is easily checked that $z=c(1/2)$ is a midpoint for $x$ and $y$, and that it is unique. 
\end{proof}

To prove that simple connectedness and local exponentiality together imply exponentiality, we first address the solvable and semisimple cases.

\begin{lem}
\label{lem:simply-connected-realizations}
Let $M$ be simply connected, and $(G,H,\sigma)$ be a transvection triple realizing it.
Denote $\tilde G$ the universal cover of $G$, and $\tilde\sigma:\tilde G\to \tilde G$ the lift of $\sigma$ to $\tilde G$.
Then $M\cong G/(G^\sigma)_0 \cong \tilde G/\tilde G^{\tilde\sigma}$ and $\tilde G^{\tilde\sigma}$ is connected.
\end{lem}
\begin{proof}
We know that $(G^\sigma)_0 \subset H \subset G^\sigma$.
But since $G$ is connected and $M$ is simply connected, $H$ must be connected, which yields $M\cong G/(G^\sigma)_0$.

Denote $\pi:\tilde G\to G$ the covering map.
Recall that the fixed point subgroup of an involution of a simply connected Lie group is always connected \cite[Chapter IV, Theorem 3.4]{loos_symmetric_1969}, so $\tilde G^{\tilde \sigma}$ is connected.
On the other hand, we have $M\cong \tilde G/\pi^{-1}(H)$. But $(\tilde G^{\tilde \sigma})_0$ is an open subgroup of $\pi^{-1}(H)$ and, by the same argument as in the foregoing paragraph, $\pi^{-1}(H)$ is connected. Hence $\pi^{-1}(H)=(\tilde G^{\tilde \sigma})_0=\tilde G^{\tilde \sigma}$.
\end{proof}

\begin{prop}[{\cite[Proposition 2.1]{benoist_multiplicite_1984}}]
\label{prop:Benoist}
Let $(G,\sigma)$ be a symmetric pair, with $G$ exponential. Then,
\begin{enumerate}[label=(\arabic*)]
\item the exponential map is a diffeomorphism from $\h$ onto $G^\sigma$;
\item the exponential map is a diffeomorphism from $\m$ onto $P$;
\item \label{prop:BenoistPt3} the multiplication $m$ is a diffeomorphism from $P\times G^\sigma$ onto $G$. \qedhere
\end{enumerate}
\end{prop}

\begin{prop}[{\cite[Corollary, p.\ 171]{loos_symmetric_1969}}]
\label{prop:Loos-cor-p171}
Let $M$ be a simply connected symmetric space with transvection group $G$, and let $\tilde G$ be its universal cover.
Then $G\cong \tilde G/Z(\tilde G)^\sigma$ where $Z(\tilde G)^\sigma=\{z\in Z(\tilde G) \mid \sigma(g)=g \}$.
\end{prop}

\begin{prop}
\label{prop:solvablesympair-pexp}
Let $M$ be solvable. 
If $M$ is connected, simply connected, and locally exponential, then $M$ and its transvection group $G$ are exponential.
\end{prop}
\begin{proof}
With the notation of Lemma~\ref{lem:simply-connected-realizations}, we have $M\cong \tilde G/\tilde G^{\sigma}$.
Since $\g$ is solvable and is the transvection Lie algebra, we see that $\h=[\m,\m]$ is contained in the nilradical of $\g$ so that for all $X\in\h$, $\ad X$ 
only has null eigenvalues.
By Lie's theorem, the complexification of $\g$ may be represented by upper triangular matrices, and the representatives of $\h$ must then be strictly upper triangular. This shows that the eigenvalues of an element $X=X_\h+X_\m\in\g$ are equal to those of $X_\m\in\m$, which cannot be purely imaginary. Hence $\tilde G$ is exponential by the Dixmier--Saito Theorem, and consequently $\tilde G/\tilde G^{\tilde\sigma}$ is exponential by Proposition~\ref{prop:Benoist}.

We now show that $G=\tilde G$.
By Proposition~\ref{prop:Loos-cor-p171},
all we need to see is that $Z(\tilde G)\cap \tilde G^{\tilde \sigma}$ is reduced to a point.
Let $z\in Z(\tilde G)\cap G^\sigma$. Since $\tilde G$ is exponential, there exists a unique $Z\in \g$ such that $z=e^Z$.
Since $z$ is in the center of $\tilde G$, we have
$
e^Z=ge^Zg^{-1}=e^{\Ad_gZ},
$
so that by exponentiality, $Z=\Ad_gZ$  for all $g\in \tilde G$, and $Z$ is in the center of $\g$.
But $z$ is also in $\tilde G^{\tilde \sigma}$, so that $Z\in\h$.
Since $\g$ is the transvection algebra, the action of $\h$ on $\m$ is effective, and the only central element in $\h$ is $0$.
Hence $z=e$.
\end{proof}

\begin{prop}
\label{prop:semisimplesympair-pexp}
Let $M$ be semisimple.
If $M$ is connected, simply connected and locally exponential, then it is a Riemannian symmetric space of the noncompact type. 
In particular, it is exponential.
\end{prop}
\begin{proof}
Let $(G,H,\sigma)$ be a transvection triple realizing $M$ where, by hypothesis, $G$ is semisimple.
Let $\theta$ be a Cartan involution commuting with $\sigma$ (\cite{berger_les_1957}, \cite[p.\ 153, Theorem 2.1]{loos_symmetric_1969}), and $B$ be the Killing form of $\g$. 
We have the decompositions in eigen\-spaces
\begin{align*}
\g &= \h \oplus \m \quad\text{for $\sigma$} \\
&= \k \oplus \p \quad\text{for $\theta$} \\
&= \h_\k \oplus \h_\p \oplus \m_\k \oplus \m_\p ,
\end{align*}
where the spaces of the last line are the obvious intersections of those of the first two lines.
Now if $0\neq Z\in\m_\k$, then $Z$ is in $\k$ 
so $\ad Z$ is a skewsymmetric matrix for the scalar product $B_\theta(X,Y)=-B(X,\theta Y)$.
Since the adjoint representation is faithful, $\ad Z$ has non-zero purely imaginary eigenvalues, contradicting local exponentiality. Hence $\m_\k=\{0\}$.
But then $\h=[\m,\m]=[\m_\p,\m_\p]=\h_\k$, so that $\h_\p=0$.
We have thus shown that $\sigma=\theta$.

Since $\sigma$ is a Cartan involution on $G$ and $G$ is connected,
the global polar decomposition for Riemannian symmetric spaces of the noncompact type \cite[Chapter VI, Theorem 1.1]{helgason_differential_1978} implies that $M$ is exponential.
\end{proof}

Using these two special cases, we may now prove the general result by induction on the dimension of the radical.
The argument is inspired from \cite{rozanov_exponentiality_2009}.

\begin{thm}
\label{thm:sconn-and-loc-exp-implies-exp-diffeo}
If $M$ is simply connected and locally exponential, then it is exponential.
Moreover, it is a semidirect product $M_1\ltimes M_2$ of a semisimple and a solvable exponential subspaces, and the map $\m_1\times \m_2\to M:(X,Y) \mapsto \Exp_o \frac{X}{2} \perp \Exp_o Y$ is a diffeomorphism.
\end{thm}
\begin{proof}
Let $G$ be the universal cover of the transvection group of $M$, so that $M\cong G/G^\sigma$, and let $R$ be the radical of $G$. 
If $R$ is $0$-dimensional, the theorem reduces to Proposition~\ref{prop:semisimplesympair-pexp}, so let us assume the contrary. 

Since $G$ is simply connected, there exists a global Levi decomposition $G=S\ltimes R$ ---with $S$ semisimple and $R$ the non-trivial radical, both simply connected--- which is compatible with $\sigma$ in the sense that the latter stabilizes both $S$ and $R$.

Let $N$ be a minimal $\sigma$-invariant normal subgroup of $G$. 
It is contained in $R$, and is connected since otherwise its connected component containing $e$ would be preserved by $\sigma$ and, $G$ being connected, by conjugation by $G$. 
Similarly, $N$ is Abelian, since otherwise the commutant would be a smaller connected $\sigma$-invariant normal subgroup.
And since it is a connected subgroup of a simply connected solvable Lie group, $N$ is also closed and simply connected \cite{malcev_theory_1945}.
Moreover, as the fixed point subgroup of an involution of a simply connected Lie group, $N^\sigma$ is connected, so that $N/N^\sigma$ is simply connected.

Let $\hat G=G/N$, let $\pi:G\to\hat G$ be the natural projection, $\hat \sigma$ the induced involution and $\pi':G/G^\sigma\to \hat G/\hat G^{\hat\sigma}$ the natural morphism. 
Since $G/G^\sigma$ is connected and simply connected, and $N/N^\sigma$ is connected, $\hat G/\hat G^{\hat\sigma}$ is connected and simply connected. Moreover, the eigenvalues of $\ad \hat\m$ are to be found amongst those of $\ad \m$, so that $\hat G/\hat G^{\hat\sigma}$ is locally exponential and, by the inductive assumption, is exponential. 

Let $x\in G/G^{\sigma}$. Then there exists a unique $X\in\hat\m$ such that $\Exp_{e\hat G^{\hat\sigma}}X=\pi'(x)$. Denote $F\subset G$ the complete preimage through $\pi$ of $\{\exp tX \mid t\in \fR\}$. It is a $\sigma$-invariant subgroup isomorphic to $\exp \left<X'\right> \ltimes N$ for some $X'\in \m$ such that $\pi_*X'=X$.
Since $N/N^\sigma$ is connected and simply connected, $F/F^\sigma=F/N^\sigma$ also is, and thus the latter inherits local exponentiality from $G/G^\sigma$. Since $F$ is solvable, $F/F^\sigma$ is then exponential by Proposition~\ref{prop:solvablesympair-pexp} so that, denoting $\mathfrak{f}$ the Lie algebra of $F$, there exists a unique $X''\in\mathfrak{f}\cap\m$ such that $\Exp_{eG^\sigma}X''=x$. 

We thus have a bijection $\Exp_{eG^\sigma}$ between $\m$ and $G/G^\sigma$, and the local exponentiality of $M$ shows that it is everywhere a local diffeomorphism, hence a global diffeomorphism.

We now turn to the second assertion.
Consider the natural map $p:G/G^\sigma\to (G/R)/(G/R)^\sigma\cong S/S^\sigma$, and set $M_1=S/S^\sigma$ and $M_2=R/R^\sigma$.
An element of $M_1$ may be uniquely written as $\Exp_oX=\exp (X) S^\sigma$ for some $X\in\m_1$.
The complete inverse image of $\Exp_oX$ by $p$ is $\{\exp (X).rG^\sigma \mid r\in R\} = \{\exp (X).\exp (Y) G^\sigma \mid Y\in \m_2\} = \{\Exp_o(\frac{X}{2})\perp\Exp_o(Y) \mid Y\in \m_2\}$, and the map $(X,Y)\mapsto \Exp_o(\frac{X}{2})\perp\Exp_o(Y)$ is clearly a diffeomorphism.
\end{proof}

\begin{prop}
\label{prop:connected-subspace-of-diffeotoRn-is-diffeotoRn-and-has-Gsigma-connected}
Let $M$ be diffeomorphic to $\fR^n$.
If $(G,\sigma)$ is a transvection pair, then $G^\sigma$ is connected.
Moreover, every connected symmetric subspace is diffeomorphic to $\fR^n$.
\end{prop}
\begin{proof}
Assume $G^\sigma$ is not connected.
Since the fixed points subgroup of an involution of a connected Lie group always has a finite number of connected components \cite[Chapter IV, Theorem 3.4]{loos_symmetric_1969}, $G/G^\sigma$ is then at the same time a symmetric space, and thus a manifold, and a quotient of $\fR^n$ by the (effective) action of the finite group $G^\sigma/(G^\sigma)_0$, hence an orbifold, a contradiction.
For the second assertion, by \cite[Theorem 1]{koh_affine_1965}, the maximal compact subspace of $M$ is trivial, hence any subspace of $M$ has the same property, and the same theorem yields the conclusion.
\end{proof}

We can now assemble the above results to prove Theorems~\ref{thm:A} and \ref{thm:B}.

\begin{proof}[Proof of Theorem~\ref{thm:A}]
The assertion \ref{thm:A:at-most-one-midpoint}$\Rightarrow$\ref{thm:A:exp-inj} follows from Lemma~\ref{lem:at-most-one-midpoint-implies-exp-inj},
\ref{thm:A:exp-inj}$\Rightarrow$\ref{thm:A:M-loc-exp} follows from Lemma~\ref{lem:exp-injective-implies-immersion-and-diffeotoRn},
\ref{thm:A:M-loc-exp}$\Leftrightarrow$\ref{thm:A:G-sigma-loc-exp} is proved in Lemma~\ref{lem:sympair-loc-exp-iff-M-loc-exp},
\ref{thm:A:M-loc-exp}$\Rightarrow$\ref{thm:A:exp-diffeo} is Theorem~\ref{thm:sconn-and-loc-exp-implies-exp-diffeo},
and finally \ref{thm:A:exp-diffeo}$\Rightarrow$\ref{thm:A:at-most-one-midpoint} follows from Lemma~\ref{lem:exp-diffeo-implies-unique-midpoint}.
\end{proof}

\begin{proof}[Proof of Theorem~\ref{thm:B}]
If $M$ is exponential, 
\ref{item:semisimplemidpointmap-implies-cartaninvolution}\ follows from Proposition~\ref{prop:semisimplesympair-pexp},
\ref{item:solvablemidpointmap-implies-exists-exponentialsymmetricpair}\ follows from Proposition \ref{prop:solvablesympair-pexp}, and
\ref{item:midpointmap-semidirect-product} follows from Theorem~\ref{thm:sconn-and-loc-exp-implies-exp-diffeo}.
\end{proof}

\begin{ex}
\label{ex:tangent-bundle}
The (canonical structure on the) tangent bundle of a locally exponential Lie triple system $\m$, defined by $T\m:=\m\oplus\m$ with
\begin{align*}
[(x,u),(y,v),(z,w)]
&=
([x,y,z],[u,y,z]+[x,v,z]+[x,y,w]),
\end{align*}
is also locally exponential.
Indeed, any operator $[(\cdot,\cdot),(y,v),(y,v)]$ has the form
$\left(\begin{smallmatrix}[\cdot,y,y] & 0 \\ [\cdot,v,y]+[\cdot,y,v] & [\cdot,y,y]\end{smallmatrix}\right)$, and thus has the same eigenvalues as $[\cdot,y,y]$.
This proves that the tangent bundle of a Riemannian symmetric space of the noncompact type is a non-semisimple, non-solvable, exponential symmetric space.
\exend
\end{ex}

\begin{ex}
A midpoint map is a smooth map $\gamma:M\times M\to M$ such that $s_{\gamma(x,y)}x=y$ for all $x,y\in M$.
There exist many \emph{finite} symmetric spaces with midpoint map (those spaces were called \emph{finite symsets} and studied by Nobusawa and his collaborators \cite{nobusawa_symmetric_1974}), the simplest non-trivial example being the space $M=\{x_1,x_2,x_3\}$ where the symmetry at a point interchanges the remaining two points, and the midpoint of two points yields the remaining point.
However, none of these finite spaces (apart from the 1-point space) embeds in any connected symmetric space with midpoint map.
Indeed, assume we have such an embedding and fix two points $o,x_0\in M$. Then the set $\{x_n\}_n$ defined by $x_n=\gamma(o,x_{n-1})$ is finite.
But one sees that it lies entirely on one single geodesic, so that the successive midpoints with $o$ should actually converge to $o$, contradicting the assumption.
\exend
\end{ex}

\begin{ex}
In \cite{bieliavsky_strict_2002}, in order to build non-formal deformation quantizations, a class of symmetric spaces is introduced as the first step of a potential inductive construction of all solvable symplectic symmetric spaces.
These (connected and simply connected) \emph{elementary solvable symplectic spaces} are realized by triples $(G,H,\sigma)$ where the Lie algebra of $G$ is a semi-direct product $\g=\a\ltimes\frb$ of two abelian Lie subalgebras stabilized by $\sigma$, such that there exists $\xi\in \g^*$ such that the 2-coboundary $\Omega:\g\times\g\to\fR$ defined by $\Omega(X,Y)=-\left<\xi,[X,Y]\right>$ restricts as a non-degenerate two-form on the $(-1)$-eigenspace $\m$ of $\sigma$ in $\g$. 
This implies that $\frb=\l\oplus\h$ with $\l\subset \m$, and that $\a$ and $\l$ are two lagrangians in involution for $\Omega|_{\m\times\m}$.

A quantization is built in \cite{bieliavsky_strict_2002} for all elementary solvable symplectic spaces which admit a midpoint map.
It is then proved that such a map exists if and only if the \emph{twisting map} $\phi:\a\to\a$, defined by $\Omega(\phi(A),L)=\left<\xi,\sinh(\ad A)L\right>$ for all $A\in\a$ and $L\in \l$, is a global diffeomorphism. 
It is easily seen that, if $\phi$ is a diffeomorphism then, for all $A\in \a$, $\cosh(\ad A):\l\to\l$ is invertible so that, $\a$ being abelian, $(\ad A)^2|_\m$ has no negative eigenvalues. Theorem~\ref{thm:A} implies that the converse is also true, giving a simple characterization of the spaces for which the constructions in \cite{bieliavsky_strict_2002} apply.
\exend
\end{ex}

\begin{cor}
\label{cor:subspaces-intersections-quotients-of-exponential-spaces}
Let $M$ be exponential. Then 
\begin{enumerate}[label={(\arabic*)}]
\item \label{cor:subspaces-intersections-quotients-of-exponential-spaces-1} a symmetric subspace $N$ is exponential if and only if midpoints of points in $N$ are in $N$, if and only if $N$ is connected;
\item \label{cor:subspaces-intersections-quotients-of-exponential-spaces-2} the intersection of any two exponential symmetric subspaces is empty or exponential;
\item \label{cor:subspaces-intersections-quotients-of-exponential-spaces-3} the quotient $M/N$ by a closed normal symmetric subspace is exponential if and only if $N$ is exponential.
\end{enumerate}
\end{cor}
\begin{proof}
\begin{enumerate}[label={(\arabic*)}]
\item If $N$ is exponential, then it obviously satisfies the other two assertions.

Assume that midpoints of points in $N$ are in $N$. 
Let $x,y\in N$. 
There is a unique geodesic $c:[0,1]\to M$ joining $x$ and $y$ in $M$.
By taking recursively the midpoints of $x$ and $y$ and of their successive midpoints, we obtain that the dyadic geodesic $c(\left\{\frac{n}{2^m}\mid m\in\fN, 0\leq n\leq 2^m\right\})$ is in $N$, and by density, that the whole geodesic $c([0,1])$ is in $N$. Hence, $N$ is exponential.

Assume now that it is connected. 
Then it is simply connected by Proposition~\ref{prop:connected-subspace-of-diffeotoRn-is-diffeotoRn-and-has-Gsigma-connected}, and it is locally exponential since $M$ is. Hence it is exponential by Theorem~\ref{thm:A}.

\item Assume the intersection $N=N_1\cap N_2$ is not empty. Then any two points in $N$ are joined by a unique geodesic in $M$ which, by 
assumption, lies in both $N_1$ and $N_2$.

\item 
Every eigenvalue of $y\mapsto[y,x,x]$ for $x\in \m/\n$ is also an eigenvalue of $Y\mapsto[Y,X,X]$ for some $X\in \m$, so $M/N$ is locally exponential.
On the other hand, it is clear that $M/N$ is simply connected if and only if $N$ is connected. \qedhere
\end{enumerate}
\end{proof}

We conclude this section by giving a characterization of the exponentiality of solvable symmetric spaces in terms of \ref{ex:2dim-solvable-non-exponential}.
To this end, we first adapt the notions of Jordan--H\"older series and roots to modules of (solvable) Lie triple systems.

If $\m$ is a Lie triple system, an \emph{$\m$-module} is a vector space $V$ together with a structure of Lie triple system on $\m\oplus V$ such that $\m$ is a subsystem and $V$ is an abelian ideal (an ideal is abelian if the triple bracket is zero whenever two of its arguments are in that ideal).
An $\m$-module is \emph{simple} if it has no non-zero proper submodule.

Let $\m$ be a solvable Lie triple system over $\fC$. Then any simple $\m$-module over $\fC$ is of dimension 1.
Indeed, let $V$ be such a module and let $\g=\m+[\m,\m]$ (respectively, $\g_V=(\m+V)+[\m+V,\m+V]=\g+(V+[\m,V])$) be the standard embedding of $\m$ (respectively, of $\m+V$).
Let $\h$ be a minimal ideal of $\g_V$ contained in the abelian ideal $V+[\m,V]$.
It is of dimension 1.
And $\h+\sigma\h=V+[\m,V]$ since $\h+\sigma\h\subset V+[\m,V]$ by $\sigma$-invariance and $\h+\sigma\h\supset V+[\m,V]$ since otherwise $(\m+V)\cap(\h+\sigma\h)$ would be a proper submodule of $V$.
So either $\h=\sigma\h$ or $\h\cap \sigma\h=\{0\}$.
In both cases, $\dim V=1$.

Let $\m$ be a solvable Lie triple system over $\fR$. Then any simple $\m$-module over $\fR$ is of dimension 1 or 2.
Indeed, let $V$, $\g$, $\g_V$ and $\h$ be as in the foregoing paragraph.
Then $\h$ is of dimension 1 or 2, and either $\h=\sigma\h$ or $\h\cap \sigma\h=\{0\}$, which yields the following four possibilities:
\begin{enumerate}[label=(S\arabic*),leftmargin=*]
\item $\h=\sigma\h$ and $\dim \h=1$: $\dim V=1$ and $V$ is central in $\m+V$,
\item \label{S2} $\h=\sigma\h$ and $\dim \h=2$: $\dim V=1$ and for all $X\in \m$, $[\cdot,X,X]|_V$ has real non-positive eigenvalue,
\item $\h\cap\sigma\h=\{0\}$ and $\dim \h=1$: $\dim V=1$ and for all $X\in \m$, $[\cdot,X,X]|_V$ has real non-negative eigenvalue,
\item \label{S4} $\h\cap\sigma\h=\{0\}$ and $\dim \h=2$: $\dim V=2$ and for all $X\in \m$, $[\cdot,X,X]|_V$ has null or non-real eigenvalues.
\end{enumerate}
In the last three cases, there exists at least a generator $X\in\m$ such that $[\cdot,X,X]|_V$ has non-zero eigenvalues.

A \emph{Jordan--H\"older series} for an $\m$-module $V$ is a decreasing sequence $V=V_0\supset V_1 \supset \dots \supset V_p=\{0\}$ of submodules of $V$ such that each quotient $V_k/V_{k+1}$ is simple.
Every module for a solvable Lie triple system $\m$ has a Jordan--H\"older series.
Indeed, let $\m$ be solvable and let $V$ be an $\m$-module of dimension $n$. 
Since $\m$ is solvable, $\m+V$ also is, and so is its standard embedding $\g_V$.
Let $\h$ be a minimal $\sigma$-invariant ideal of $\g_V$ contained in $V+[\m,V]$. 
Then $W=(\m+V)\cap\h$ is a simple ideal of $\m+V$ contained in $V$, i.e., a simple submodule of $V$.
Assume by induction that any $\m$-module of dimension not greater than $n-1$ has a Jordan--H\"older series, and let $V'_0=V'\supset V'_1\supset\dots\supset V'_{n-1}=\{0\}$ be such a series for $V'=V/W$.
If $\pi:V\to V/W$ denotes the canonical projection, then setting $V_k=\pi^{-1}(V'_k)$ for $k=0,\dots,n-2$, and $V_{n-1}=W$, $V_n=\{0\}$ defines a Jordan--H\"older series for $V$.

The Jordan--H\"older theorem holds for $\m$-modules since they can be viewed as groups with operators \cite[pp. 30--31, 41]{bourbaki_algebra_1998} $(V,\m\times\m)$ with $V$ seen as an abelian group, the action being $(\m\times \m)\times V\to V:((X,Y),v)\mapsto[v,X,Y]$.

Given a solvable Lie triple system $\m$ over $\fR$, an $\m$-module $V$ of dimension $n$ and a Jordan--Hölder series $V=V_0\supset V_1 \supset \dots \supset V_p=\{0\}$ for $V$, the \emph{weights} of $V$ are defined as follows.
Let $\m_\fC$, $V_\fC$ and $V_{k,\fC}$ be the complexifications of $\m$, $V$ and $V_k$, $k=0,\dots,p$ respectively.
If, for some $k$, $V_{k,\fC}/V_{k+1,\fC}$ is not simple as an $\m_\fC$-module (this will only happen in the case \ref{S4} above), insert a submodule $V'_k$ between $V_{k,\fC}$ and $V_{k+1,\fC}$ such that $V_{k,\fC}/V'_k$ and $V'_k/V_{k+1,\fC}$ are simple.
Doing this for all $k$ yields a Jordan--Hölder series $\tilde V_l$, $l=0,\dots,n-1$ for $V_\fC$.
For each $l$, by restricting to $\tilde V_l$ and passing to the quotient $\tilde V_l/\tilde V_{l+1}$, the map $\m_\fC\to\End(V_\fC):X\mapsto [\cdot,X,X]$ defines a quadratic form $\tilde \phi_l:\m_\fC\to\End(\tilde V_l/\tilde V_{l+1})\cong \fC$.
Restricting the forms $\tilde\phi_l$ to $\m\subset \m_\fC$ yields the \emph{weights} $\phi_l:\m\to\fC$ of $V$ as an $\m$-module.
Up to the order, they are independent of the choice of Jordan--Hölder series for $V$ and of the choice of submodules $V'_k$.
When the $\m$-module is $\m$ itself with the module structure described in Example~\ref{ex:tangent-bundle}, the weights are called the \emph{roots} of $m$.

\begin{proof}[Proof of Theorem~\ref{thm:D}]
If a factor space of $M$ has a subspace isomorphic to \ref{ex:2dim-solvable-non-exponential}
then, by Corollary~\ref{cor:subspaces-intersections-quotients-of-exponential-spaces}, $M$ is not exponential.

Conversely, assume that $M$ is not exponential.
Let $\m=\m_0\supset \m_1 \supset \dots \supset \m_p=\{0\}$ be a Jordan--Hölder series for $\m$ and $\phi_l$, $l=0,\dots,\dim \m-1$ be the corresponding roots as constructed above.
Then, since $\m$ is not locally exponential, there exists a root $\phi_l$ of $\m$ corresponding to a simple quotient $\m_k/\m_{k+1}$ of the kind \ref{S2} above.
Let $\hat X\in\m$ be such that $\phi_l(\hat X)=-1$, let $\hat Y\in\m_k\setminus\m_{k+1}$, and let $X,Y$ be their respective images in $\m/\m_{k+1}$.
Then, in $\m/\m_{k+1}$, we have $[Y,X,X]=-Y$, and moreover $[Y,X,Y]=0$ since $\m_k/\m_{k+1} + [\m,\m_k/\m_{k+1}]$ is abelian.
Hence $X$ and $Y$ span a Lie triple system 
isomorphic to that of \ref{ex:2dim-solvable-non-exponential}, which by Proposition~\ref{prop:connected-subspace-of-diffeotoRn-is-diffeotoRn-and-has-Gsigma-connected} implies the result.
\end{proof}

\section{Double triangles and solvable spaces}
\label{sec:double-triangles-and-solvable-spaces}

In this section, $M$ will denote an exponential symmetric space with base point $o$, and $\gamma:M\times M\to M$ its \emph{midpoint map}: the smooth map sending two points to their unique midpoint.
\begin{defn}[\cite{qian_groupoids_1997}]
The map $\gamma_3:M\times M \times M \to M\times M \times M$ is defined by
\begin{align*}
\gamma_3(x,y,z)=(\gamma(z,x),\gamma(x,y),\gamma(y,z))
\end{align*}
for all $x,y,z\in M$.
A triple $(a,b,c)$ such that $\gamma_3(a,b,c)=(x,y,z)$ is called a \emph{double triangle of $(x,y,z)$}.
\end{defn}

\begin{ex}
\label{ex:hyperbolic-space-not-all-double-triangles}
It was shown in \cite{tuynman_areas_1999,rios_variational_2004,rios_weyl_2008} that there exist triangles $(x,y,z)$ in the hyperbolic plane that have no double triangle, so that $\gamma_3$ is not a diffeomorphism. 
If a given triangle has a double triangle, however, the latter is unique.
Seeing the hyperbolic plane as the $x_3>0$ sheet of the two-sheeted hyperboloid $(x_3)^2-(x_1)^2-(x_2)^2=1$ in $\fR^3$, the condition \cite{tuynman_areas_1999} for the existence of a double triangle of a triangle $(x,y,z)$ is that $\det(xyz)^2<1$, where $xyz$ denotes the $3\times 3$ matrix whose columns are the vectors $x$, $y$ and $z$.
\exend
\end{ex}

\begin{thm}
\label{thm:solvable-symmetric-space-gamma3-diffeomorphism}
If $M$ is solvable,
then $\gamma_3$ is a diffeomorphism. 
Moreover, for any odd natural number $n$, every $n$-gon has a unique double in the sense that the obvious generalization $\gamma_n:M^n\to M^n$ is a diffeomorphism.
\end{thm}
\begin{proof}
Let $(G,\sigma)$ be a transvection pair, so that $M\cong G/G^\sigma$ and $G$ is exponential.
We will prove the result by induction on the dimension of $G$. 

If $G$ is $1$-dimensional, then $M$ is isomorphic to $\fR$ with the usual symmetries $s_xy=2x-y$. Hence, $\delta_3(x,y,z)=(y-z+x,z-x+y,x-y+z)$ is a smooth inverse of $\gamma_3$.

Assume now that the dimension of $G$ is greater than $1$, and let $N$ be a minimal $\sigma$-invariant normal subgroup of $G$. 
As in the proof of Theorem~\ref{thm:sconn-and-loc-exp-implies-exp-diffeo}, we can see that $N$ is a closed abelian connected and simply connected normal subgroup, and by Proposition~\ref{prop:connected-subspace-of-diffeotoRn-is-diffeotoRn-and-has-Gsigma-connected}, $N/N^\sigma$ is also simply connected.

Set $\hat G=G/N$.
Since $\hat G/\hat G^{\hat\sigma}$ is contractible, the canonical morphism $\pi:G/G^\sigma\to \hat G/\hat G^{\hat\sigma}$ (which is a fibre bundle) admits a smooth global section $s$ \cite[pp.\ 25, 55--56]{steenrod_topology_1951}.
Moreover, since $G/G^\sigma$ is exponential, there is a smooth global section $t$ of the canonical projection $\rho:G\to G/G^\sigma$ with values in $G_\sigma=\{ g\sigma(g)^{-1} \mid g\in G \}\subset G$ (it is given by $t(\Exp_oX)=Q(\Exp_o\frac{X}{2})$ where $Q$ is the quadratic representation defined in the introduction).

To summarize, we have the following commutative diagram, where the arrows on the two upper lines are group morphisms, and those on the lower line are symmetric space morphisms:
\begin{displaymath}
\xymatrix{
  N^\sigma \ar[d] \ar[r] 
  & 
  G^\sigma \ar[d] \ar[r]
  &
  G^\sigma/N^\sigma \ar[d]
  \\
  N \ar[d] \ar[r]
  &
  G \ar[d]_{\rho} \ar[r]
  & 
  G/N \ar[d] 
  \\
  N/N^\sigma \ar[r]
  &
  G/G^\sigma \ar[r]_\pi \ar@{.>}@/_/[u]_{t}
  &
  \hat G/\hat G^{\hat\sigma} \ar@{.>}@/_/[l]_s
}
\end{displaymath}

Thus the composition $t\circ s$ gives a map $\hat G/\hat G^{\hat\sigma}\to G$, and for each $x\in \hat G/\hat G^{\hat\sigma}$ the set $\{ t(s(x)).n \mid n\in N \}$ is mapped through $\rho$ onto the fiber $\pi^{-1}(x)$.

Moreover, since $N$ is abelian connected and simply connected, it is exponential so that by Proposition~\ref{prop:Benoist} there exists a smooth section $r$ of $N\to N/N^\sigma$ with values in the Lie subgroup $N'=\{ n\in N \mid \sigma(n)=n^{-1} \}$ of $N$.

With the help of the above constructions, we may uniquely represent any element $x\in G/G^\sigma$ as 
\begin{align}
\label{eq:proof-solvable-double-triangles-repr-x-gn}
\text{$x=gnG^\sigma$ where $g=t(s(\pi(x)))$ and $n=r(g^{-1}x)\in N'$.}
\end{align}

Let now $x_1,x_2,x_3\in G/G^\sigma$. 
By the induction assumption, there exists a unique $\hat x\in \hat G/\hat G^{\hat\sigma}$ such that 
$
s_{\pi(x_3)}s_{\pi(x_2)}s_{\pi(x_1)}\hat x=\hat x, 
$
and $\hat x$ is a smooth function of $(\pi(x_1),\pi(x_2),\pi(x_3))$.
We will build a unique $x\in G/G^\sigma$ depending smoothly on $(x_1,x_2,x_3)$ and such that 
\begin{align}
\label{eq:inductive-solution-double-triangle}
s_{x_3}s_{x_2}s_{x_1}x=x.
\end{align}
If $x$ is a solution of \eqref{eq:inductive-solution-double-triangle}, it must be such that $\pi(x)=\hat x$ by uniqueness of $\hat x$.
Hence, writing $x=gnG^\sigma$ with $g=t(s(\hat x))$ and $n\in N'$, solving equation \eqref{eq:inductive-solution-double-triangle} amounts to solving an equation on $n$.
Writing $x_i=g_iG^\sigma$ for some $g_i\in G$, we have
\begin{align}
\label{eq:inductive-solution-double-triangle-last-line}
s_{x_3}s_{x_2}s_{x_1} (gnG^\sigma)
&= 
g_3\sigma(g_3^{-1}g_2)g_2^{-1}g_1\sigma(g_1^{-1}g)
\sigma(n) G^\sigma 
= \tilde gn^{-1}G^\sigma , 
\end{align}
where $\tilde g=g_3\sigma(g_3^{-1}g_2)g_2^{-1}g_1\sigma(g_1^{-1}g)$.
But by \eqref{eq:proof-solvable-double-triangles-repr-x-gn} and the definitions of $g$ and $\hat x$, we have $\tilde gG^\sigma=g\tilde n G^\sigma$ for a unique $\tilde n \in N'$. 
Hence there further exists a unique $k\in G^\sigma$ such that $\tilde g=g\tilde n k$, and the right hand side in equation~\eqref{eq:inductive-solution-double-triangle-last-line} is equal to
$
g\C_{k}(n^{-1})\tilde n G^\sigma 
$,
where $\C_k(n^{-1})\in N'$ since $k\in G^\sigma$ and $n\in N'$, and where we used the fact that $N'$ is abelian.
This will be equal to $gnG^\sigma$ if and only if
$n=\C_k(n^{-1})\tilde n$. 
Replacing the $n$'s by their $\log$, we get
$
\left(\Id + \Ad_{k} \right)(n) = \tilde n .
$
Now $G$ is exponential so $k=e^X$ with $X\in \k$ nilpotent, so that
$\Id+\Ad_k$ is invertible and its inverse is polynomial in $\ad X$.
We thus have a unique solution $n=\left(\Id+\Ad_k\right)^{-1}\tilde n$ which is smooth in $(x_1,x_2,x_3)$, and consequently a unique solution $x=t(s(\hat x))nG^\sigma$ of the equation $s_{x_3}s_{x_2}s_{x_1}x=x$ with the same property.
Hence $\gamma_3$ is invertible and has a smooth inverse.
As $\gamma_3$ itself is smooth, we have a diffeomorphism.

For the second assertion of the theorem, just notice that the argument above only relied on the fact that there was an odd number of symmetries in equation~\eqref{eq:inductive-solution-double-triangle}.
\end{proof}

\begin{prop}
\label{prop:connected-symmetric-subspaces-inherit-gamma3-diffeomorphism}
If $\gamma_3$ is a diffeomorphism and if $N$ is a connected symmetric subspace of $M$, then $\gamma_3|_{N^3}$ is a diffeomorphism $N^3\to N^3$.
\end{prop}
\begin{proof}
Since $N$ is connected, by Corollary~\ref{cor:subspaces-intersections-quotients-of-exponential-spaces} \ref{cor:subspaces-intersections-quotients-of-exponential-spaces-1}, $\gamma$ restricts to $N$ as a midpoint map, so that $\gamma_3(N^3)\subset N^3$. 
Obviously, $\gamma_3(N^3)$ is open in $N^3$.
But $N$ is closed in $M$, so $\gamma_3(N^3)$ is closed in $M^3$ and thus in $N^3$. 
\end{proof}

\begin{proof}[Proof of Theorem~\ref{thm:C}]
As in Lemma~\ref{lem:simply-connected-realizations}, let us realize $M$ as $\tilde G/\tilde G^\sigma$, where $\tilde G$ is the universal cover of the transvection group of $M$.
Now there exists a global Levi decomposition $\tilde G=S\ltimes R$, with $S$ and $R$ stable under $\sigma$, $S$ semisimple and $R$ the radical.
Thus $S/S^\sigma$ is a semisimple symmetric subspace of $M$.

Let us prove that for any semisimple exponential symmetric space $M'$, $\gamma_3$ is not surjective so that, by Proposition~\ref{prop:connected-symmetric-subspaces-inherit-gamma3-diffeomorphism}, $S/S^\sigma$ must be reduced to a point and $M$ must be solvable.
To realize $M'$, by Theorem~\ref{thm:A} it is sufficient to consider $(G,\theta)$ a Riemannian symmetric pair of the noncompact type.
Then there exists a symmetric subspace of $M'$ isomorphic to the hyperbolic plane
\cite[Chapter IX]{helgason_differential_1978}
for which, by Example~\ref{ex:hyperbolic-space-not-all-double-triangles}, $\gamma_3$ is not surjective. 
Proposition~\ref{prop:connected-symmetric-subspaces-inherit-gamma3-diffeomorphism} then implies that so is the case for $M'$.

The converse implication of the theorem is the content of Theorem~\ref{thm:solvable-symmetric-space-gamma3-diffeomorphism}.
\end{proof} 


\bibliographystyle{abbrv}
\bibliography{PaperMidpoints}

\end{document}